\documentclass[12pt]{article}
\usepackage{amssymb}
\usepackage{mathrsfs}
\usepackage{amsthm}
\usepackage{amsmath}
\usepackage{graphicx}
\usepackage{tikz}
\usepackage{lmodern}
\usepackage{enumerate}

\usepackage{latexsym,amsmath,amssymb,amsfonts,epsfig,graphicx,cite,psfrag}
\usepackage{eepic,color,colordvi,amscd}

\newtheorem{thm}{Theorem}

\newtheorem{lem}[thm]{Lemma}
\newtheorem{coro}[thm]{Corollary}

\newcommand{\flo}[1]{{\left\lfloor{#1}\right\rfloor}}

\begin{document}

\title{Cycles with consecutive odd lengths}
\author{Jie Ma\thanks{Email: jiemath@andrew.cmu.edu.}\\
Department of Mathematical Sciences, \\Carnegie Mellon University, Pittsburgh, PA 15213.}

\date{}

\maketitle

\begin{abstract}
It is proved that there exists an absolute constant $c>0$ such that for every natural number $k$,
every non-bipartite 2-connected graph with average degree at least $ck$ contains $k$ cycles with consecutive odd lengths.
This implies the existence of the absolute constant $d>0$ that every non-bipartite 2-connected graph
with minimum degree at least $dk$ contains cycles of all lengths modulo $k$,
thus providing an answer (in a strong form) to a question of Thomassen in \cite{Th83}.
Both results are sharp up to the constant factors.
\end{abstract}

\section{Introduction}
The research of cycles has been fundamental since the beginning of graph theory.
One of various problems on cycles which have been considered is the study of cycle lengths modulo a positive integer $k$.
Burr and Erd\H{o}s \cite{Erd76} conjectured that for every odd $k$, there exists a constanct $c_k$ such that
every graph with average degree at least $c_k$ contains cycles of all lengths modulo $k$.
In \cite{B77}, Bollob\'as resolved this conjecture by showing that $c_k= 2[(k+1)^k-1]/k$ suffices.
Thomassen \cite{Th83,Th88} strengthened the result of Bollob\'as by proving that for every $k$ (not necessarily odd),
every graph with minimum degree at least $4k(k+1)$ contains cycles of all even lengths modulo $k$,
which was improved to the bound $2k-1$ by Diwan \cite{D10}.
Note that in case $k$ is even, any integer congruent to $l$ modulo $k$ has the same parity with $l$,
and thus we can not expect that there are cycles of all lengths modulo $k$ in bipartite graphs (even with sufficient large minimum degree).
On the other hand, Thomassen \cite{Th83} showed that for every $k$ there exists a least natural number $f(k)$ such that
every non-bipartite 2-connected graph with minimum degree at least $f(k)$ contains cycles of all length modulo $k$.
Here the 2-connectivity condition can not be further improved, as one can easily construct a non-bipartite connected graph
with arbitrary large minimum degree but containing a unique (also arbitrary) odd cycle.
Thomassen \cite{Th83} remarked that the upper bound for $f(k)$ obtained by him is perhaps ``far too large'' and asked if $f(k)$ can be bounded above by a polynomial.

Bondy and Vince \cite{BV} resolved a conjecture of Erd\H{o}s by showing that every graph with minimum degree at least 3 contains two cycles whose lengths differ by one or two.
Verstra\"ete \cite{V00} proved that if graph $G$ has average degree at least $8k$ and even girth $g$ then there are $(g/2-1)k$ cycles of consecutive even lengths in $G$.
In an attempt to extend the result of Bondy and Vince, Fan \cite{Fan02} showed that every graph with minimum degree at least $3k-2$ contains $k$ cycles of consecutive even lengths or consecutive odd lengths.
Sudakov and Verstra\'ete proved \cite{SV08} that if graph $G$ has average degree $192(k+1)$ and girth $g$ then there are $k^{\flo{(g-1)/2}}$ cycles of consecutive even lengths in $G$,
strengthening the above results in the case that $k$ and $g$ are large.
It is natural to ask if one can pursue the analogous result for odd cycles.
In this paper, we show that this indeed is the case by the following theorem.

\begin{thm}\label{thm:main}
There exists an absolute constant $c>0$ such that for every natural number $k$,
every non-bipartite 2-connected graph $G$ with average degree at least $ck$ and girth $g$ contains at least $k^{\flo{(g-1)/2}}$ cycles of consecutive odd lengths.
\end{thm}

\noindent We point out that the non-bipartite condition here is necessary and the 2-connectivity conditions can not be improved.
In view of the {\it Moore Bound}, our lower bound on the number of cycles is best possible (up to constant factor) for infinitely many integers $k$ when $g\le 8$ or $g=12$.
And more generally, the well-known conjecture that the minimal order of graphs with minimal degree $k$ and girth $g$ is $O(k^{\flo{(g-1)/2}})$
would imply that our result gives the correct order of magnitude for other values of $g$.

Let $G$ be a graph as in Theorem \ref{thm:main}. It is clear that there are at least $k$ cycles of consecutive odd lengths in $G$,
which assures that $G$ contains cycles of all odd lengths modulo $k$ (whenever $k$ is even or odd).
Together with the aforementioned result of Diwan on cycles of all even lengths modulo $k$,
we answer the question of Thomassen by improving the upper bound of $f(k)$ to a linear function by the following corollary.

\begin{coro}\label{coro:cycle_mod_k}
There exists an absolute constant $d>0$ such that for every natural number $k$,
every non-bipartite 2-connected graph with minimum degree at least $dk$ contains cycles of all lengths modulo $k$.
\end{coro}

\noindent This bound is sharp up to the constant factor: the complete graph $K_{k+1}$ contains cycles of all lengths but $2$ modulo $k$ and thus shows that $f(k)\ge k+1$.

\medskip

All graphs considered are simple and finite.
Let $G$ be a graph. We denote the number of vertices in $G$ by $|G|$, the vertex set by $V(G)$, the edge set by $E(G)$ and the minimum degree by $\delta(G)$.
If $S\subset V(G)$, then $G-S$ denotes the subgraph of $G$ obtained by deleting all vertices in $S$ (and all edges incident with some vertex in $S$).
If $S\subset E(G)$, then $G-S$ is obtained from $G$ by deleting all edges in $S$.
Let $A$ and $B$ be subsets of $V(G)$. An {\it $(A,B)$-path} in $G$ is a path with one endpoint in $A$ and the other in $B$.
If $A$ only contains one vertex $a$, then we simply write $(A,B)$-path as $(a,B)$-path.
We say a path $P$ is {\it internally disjoint} from $A$, if no vertex except the endpoints in $P$ is contained in $A$.

The rest of paper is organized as follows.
In next section, we establish Theorem \ref{thm:main} based on the approach of \cite{SV08}.
The last section contains some remakes and open problems.
We make no effort to optimize the constants in proofs and instead aim for simpler presentation.

\section{The proof}
Before processing, we state the following functional lemma from \cite{V00},
which will be applied multiple times and become essential in the proof of our main theorem.
\begin{lem} {\rm (Verstra\"ete \cite{V00})}
\label{lem:A-Bpaths}
Let $C$ be a cycle with a chord, and let $(A,B)$ be a nontrivial
partition of $V(C)$. Then $C$ contains $(A,B)$-paths of every length less than $|C|$, unless $C$ is
bipartite with bipartition $(A,B)$.
\end{lem}

\begin{proof}[Proof of Theorem \ref{thm:main}]
We shall show that it suffices to use $c=456$. Let $G$ be a non-bipartite 2-connected graph with average degree at least $456\cdot k$ and girth $g$.
Our goal is to show that $G$ contains $t:=k^{\flo{(g-1)/2}}$ cycles of consecutive odd lengths.
As it holds trivially when $k=1$, we assume that $k\ge 2$.

Let $G_b$ be a bipartite subgraph of $G$ with the maximum number of edges.
It is easy to see that $G_b$ is a connected spanning subgraph of $G$ with average degree at least $228k$.
Since $G$ is non-bipartite, there exists an edge $xy\in E(G)$ such that
both $x$ and $y$ lie in the same part of the bipartition of $G_b$.
Let $T$ be the breadth first search tree in $G_b$ with root $x$,
and let $L_i$ be the set of vertices of $T$ at distance $i$ from its root $x$ for $i\ge 0$ (so $L_0:=\{x\}$).
As $T$ is also a spanning tree, it follows that $V(G)=V(G_b)=\cup_{i\ge 0} L_i$.
For any two vertices $a$ and $b$ in the tree $T$, we denote $T_{ab}$ by the unique path in $T$ with endpoints $a$ and $b$.
By the choice of edge $xy$, clearly the vertex $y$ lies in $L_{2l}$ for some integer $l\ge 1$,
therefore the cycle $D:=T_{xy}\cup xy$ is an odd cycle in $G$.

By the definition of $T$, every edge of $G_b$ joints one vertex in $L_i$ to the other in $L_{i+1}$ for some $i\ge 0$.
Thus, we have
$$\sum_{i} e(L_i,L_{i+1})=e(G_b)\ge 114k\cdot |G_b|=114k\cdot \sum_i |L_i|\ge 57k\cdot \sum_{i}\left(|L_i|+|L_{i+1}|\right).$$
So there must exist some $i\ge 0$ such that the induced (bipartite) subgraph $G_i:=G_b[L_i\cup L_{i+1}]$ has average degree at least $114k$.
We now use the following lemma in \cite{SV} to find a long cycle with at least one chord (in fact with many chords) in $G_i$.

\vskip 2mm

\begin{lem} {\rm (Sudakov and Verstra\"ete \cite{SV})}
Let $G$ be a graph of average degree at least $12(d+1)$ and girth $g$. Then $G$ contains a cycle $C$ with at least $\frac{1}{3}d^{\flo{(g-1)/2}}$ vertices of degree at least $6(d+1)$, each of which has no neighbors in $G-V(C)$.
\end{lem}

\vskip 2mm

Note that $t=k^{\flo{(g-1)/2}}$. Since $114k\ge 12(9k+1)$ for $k\ge 2$ and the girth of $G_i$ is at least $g$, this lemma shows that there is a cycle $C$ with a chord in $G_i$, satisfying
\begin{equation}\label{equ:|C|}
|C|\ge \frac{1}{3}\cdot (9k)^{\flo{(g-1)/2}}\ge 2(t+1).
\end{equation}
We notice that $(V(C)\cap L_i, V(C)\cap L_{i+1})$ is the unique bipartition of $C$ (this is for the use of Lemma \ref{lem:A-Bpaths} later).

Let $T'$ be the minimal subtree of $T$ whose set of leaves is precisely $V(C)\cap L_i$, and let $z$ be the root of $T'$
(i.e., the one at the shortest distance from $x$). By the minimality of $T'$, $z$ has at least two branches of $T'$.
Let the depth of $T'$ (i.e., the distance between its root $z$ and its leaves) be $j$.

Recall that $D$ is an odd cycle in $G$ consisting of $T_{xy}$ and edge $xy$, where $y\in L_{2l}$.
Depending on whether $V(D)$ interests $V(C)\cup V(T')-\{z\}$ or not, we distinguish the following two cases.

\bigskip

{\bf Case 1.} $V(D)$ and $V(C)\cup V(T')-\{z\}$ are disjoint.

\vskip 1mm

In this case, the tree $T$ contains a path $Z$ from $z$ to $V(D)$, which is internally disjoint from $V(C)\cup V(T')$.
Note that $G$ is 2-connected, so there are two disjoint $(V(C)\cup V(T'), V(D))$-paths, say $P$ and $Q$, in $G$,
which are internally disjoint from $V(C)\cup V(T')\cup V(D)$.
Routing $P,Q$ through $Z$ if necessary, we may assume that $P$ is from $z$ to $p\in V(D)$,
and $Q$ is from $w\in V(C)\cup V(T')-\{z\}$ to $q\in V(D)-\{p\}$.

Base on the location of $w$, we divide the remainder of this case into two parts.
Let us first consider when $w\in V(T')-\{z\}$.
Let $A$ be the set of all leaves in the subtree of $T'$ with root $w$, and let $B:=V(C)-A$.
As $w\neq z$, we see $B^*:=V(C)\cap L_i-A$ is nonempty,
which shows that $(A,B)$ is not the bipartition of $C$.
By Lemma \ref{lem:A-Bpaths} and the equation \eqref{equ:|C|}, $C$ contains $(A,B)$-paths of all even lengths up to $2t+1$,
all of which in fact are $(A,B^*)$-paths because $C$ is bipartite.
To find $t$ cycles of consecutive odd lengths, now it is enough to show that for any $a\in A$ and $b\in B^*$,
there exists an $(a,b)$-path in $G-E(C)$ with a fixed odd length.
To see this, first note that paths $T'_{aw}$ and $T'_{bz}$ are disjoint and of fixed lengths;
since $D$ is an odd cycle, one can choose a $(p,q)$-path $R$ in $D$ (out of two choices) such that
the $(a,b)$-path $T'_{aw}\cup Q\cup R\cup P\cup T'_{zb}$ in $G-E(C)$ is of a fixed odd length.

Now we consider the situation when $w\in V(C)-V(T')$.
So $w\in V(C)\cap L_{i+1}$, and clearly $(\{w\},V(C)-\{w\})$ is not a bipartition of $C$.
By Lemma \ref{lem:A-Bpaths} and the equation \eqref{equ:|C|},
$C$ contains $(w,V(C)-\{w\})$-paths of all odd lengths up to $2t+1$,
all of which are $(w, V(C)\cap L_i)$-paths because $C$ is bipartite.
For any $u\in V(C)\cap L_i$, the length of path $T'_{uz}$ is fixed, that is the depth $j$ of $T'$.
Therefore, $C\cup T'$ contains $(z,w)$-paths of all lengths in $\{1+j,3+j,\ldots,2t+1+j\}$.
Since $D$ is an odd cycle, similarly as in the last paragraph one can choose a $(p,q)$-path $R$ in $D$ such that
the $(z,w)$-path $P\cup R\cup Q$ in $G-E(C\cup T')$ is of the same parity as $j$.
Putting the above paths together, we see that $G$ contains at least $t$ cycles of consecutive odd lengths.
This completes the proof of Case 1.

\bigskip

{\bf Case 2.} $V(D)$ intersects $V(C)\cup V(T')-\{z\}$.

\vskip 1mm

Let $w\in V(D)\cap \left(V(C)\cup V(T')-\{z\}\right)$ be the vertex such that $T_{wy}$ is the shortest path among all choices of $w$.
Note that now $T_{wz}=T'_{wz}$ is a subpath in the cycle $D$.

Again we distinguish on the location of $w$. First we consider when $w\in V(T')-\{z\}$.
Let $A$ be the set of all leaves in the subtree of $T'$ with root $w$, and $B^*:=V(C)\cap L_i-A$.
By the same proof as in the second paragraph of Case 1,
we conclude that $C$ contains $(A,B^*)$-paths of all even lengths up to $2t+1$.
We also notice that $D$ consists of two $(w,z)$-paths $T'_{wz}$ and $P:=D-T'_{wz}$, whose lengths are of opposite parities.
For any $a\in A$ and $b\in B^*$, $T'_{aw}\cup T'_{wz}\cup T'_{zb}$ forms an $(a,b)$-walk with a fixed even length,
that is twice of the depth of $T'$.
This suggests that $T'_{aw}\cup P\cup T'_{zb}$ is an $(a,b)$-path in $G-E(C)$ with a fixed odd length,
which, combining with these $(A,B^*)$-paths with lengths $2,4,\ldots, 2t$ in $C$, comprise $t$ cycles of consecutive odd lengths in $G$.

We are left with the case when $w\in V(C)-V(T')$. This shows that $w\in V(C)\cap L_{i+1}$ and $(\{w\}, V(C)-\{w\})$ is not the bipartition of $C$.
By Lemma \ref{lem:A-Bpaths} as well as the equation \eqref{equ:|C|}, $C$ contains $(w,V(C)-\{w\})$-paths of all odd lengths up to $2t+1$.
Similarly as the previous proofs, these odd paths are actually $(w,V(C)\cap L_i)$-paths.
Recall that the depth of $T'$ is $j$.
Therefore, $C\cup T'$ contains $(w,z)$-paths of all lengths in $\{1+j,3+j,\ldots, 2t+1+j\}$
and particularly the sub-path $T_{wz}$ of $D$ has length $1+j$. We know $D$ is an odd cycle,
so $D-T_{wz}$ is a $(w,z)$-path in $G-E(C\cup T')$ whose length is of the same parity as $j$.
Putting $D-T_{wz}$ and these $(w,z)$-paths in $C\cup T'$ together,
we find at least $t$ cycles of consecutive odd lengths.
This proves Case 2, finishing the proof of Theorem \ref{thm:main}.
\end{proof}

\section{Concluding remarks}
In \cite{BV} Bondy and Vince gave an infinite family of non-bipartite 2-connected with arbitrary large minimum degree but containing no two cycles whose lengths differ by one. This tells that Theorem \ref{thm:main} is sharp from another point of view.
The situation changes completely when the connectivity increases.
Fan \cite{Fan02} showed that every non-bipartite 3-connected graph with minimum degree at least $3k$ contains $2k$ cycles of consecutive lengths.
A conjecture of Dean (see \cite{DLS93}) also considered the connectivity and
asserted that every $k$-connected graph contains a cycle of length $0$ modulo $k$.
We observe that this is best possible for odd $k$ (if true), as the complete bipartite graph $K_{k-1,k-1}$ is $(k-1)$-connected but has no cycle of length $0$ modulo $k$.

Thomassen showed in \cite{Th83-girth} that graphs of minimum degree at least 3 and large girth share many properties with
graphs of large minimum degree. For example, he proved that if $G$ is a graph of minimum
degree at least 3 and girth at least $2(k^2 + 1)(3\cdot 2^{k^2+1} + (k^2 + 1)^2-1)$, then $G$ contains
cycles of all even lengths modulo $k$ (while we have seen the analogous result for graphs of large minimum degree in \cite{B77,Th83,Th88,D10}).
We have also seen that graphs of large minimum degree contain cycles of consecutive even lengths (e.g., results from \cite{SV08,V00}),
however to the best of our knowledge it is not known if there exists a natural number $g(k)$ such that
every graph of minimum degree at least 3 and girth at least $g(k)$ contains $k$ cycles of consecutive even lengths.
Similar question can be raised with respect to Theorem \ref{thm:main} as well.

In 1970s Erd\H{o}s and Simonovits \cite{ES} asked to determine the {\it chromatic profile}
$$
\delta_\chi(H,k):=\inf \{c: \delta(G)\ge c|G| \mbox{ and } H\not\subset G \Rightarrow \chi(G)\le k\}$$
for every graph $H$ (we refer interested readers to \cite{ABGKM} for related topics).
Since then, very little is known about $\delta_\chi(H,k)$ for graphs $H$ other than $K_3$.
Thomassen \cite{Th07} proved that for every $c>0$ and odd integer $l\ge 5$,
every $C_l$-free graph $G$ with minimum degree at least $c|G|$ has chromatic number $\chi(G)$ less than $(l+f(2l-8))/c$.
We conclude this paper with the following.

\begin{lem}
For arbitrary fixed odd integer $l\ge 5$, it holds that
$$\Theta\left(\frac{1}{(k+1)^{4(l+1)}}\right)\le \delta_\chi(C_l,k)\le \Theta\left(\frac{l}{k}\right).$$
\end{lem}

\begin{proof}
The upper bound can be obtained easily by combining Thomassen's result \cite{Th07} and Corollary \ref{coro:cycle_mod_k}.
We turn to the lower bound. Let $N(g,k)$ be the minimum $|G|$ over all graphs $G$ with girth at least $g$ and chromatic number at least $k$.
The proof of \cite{Mar} shows that $N(g,k)\le \Theta(k^{4g+1})$ when $k\ge 144$.

It suffices to prove $\delta_\chi(C_l,k)> k/N(l+1,k+1)$.
Let $G_0$ be a graph of minimum order with girth at least $l+1$ and chromatic number at least $k+1$, i.e., $|G_0|= N(l+1,k+1)$,
then by the minimality it holds that $\chi(G_0)=k+1$ and $\delta(G_0)\ge k$.
We then construct a graph $G$ obtained from $G_0$ by replacing every vertex with an independent set of size $t$ and every edge with a complete bipartite graph.
Clearly $G$ contains no $C_l$ (in fact there is no odd cycle of length less than $l+1$ in $G$),
where $\chi(G)=\chi(G_0)>k$ and $\delta(G)=t\cdot\delta(G_0)\ge \frac{k|G|}{N(l+1,k+1)}$.
This completes the proof.
\end{proof}

\end{document}